\documentclass[dvips, 11pt]{article}

\usepackage{hyperref}
\usepackage{amsmath, amsthm, amssymb, enumerate, graphicx}
\usepackage[margin=1in]{geometry}

\newcommand{\C}{\mathbf{C}}
\newcommand{\Z}{\mathbf{Z}}
\newcommand{\R}{\mathbf{R}}

\newcommand{\N}{\mathbf{N}}
\renewcommand{\phi}{\varphi}

\DeclareMathOperator*{\conv}{conv}
\renewcommand{\P}{{\mathcal P}}
\newcommand{\F}{\mathcal{F}}
\newcommand{\T}{\mathcal{T}}
\DeclareMathOperator{\vol}{vol}

\theoremstyle{plain}
\newtheorem{theorem}{Theorem}
\newtheorem{corollary}[theorem]{Corollary}

\newtheorem{lemma}[theorem]{Lemma}

\theoremstyle{definition}
\newtheorem{example}[theorem]{Example}

\date{June 27, 2008}
\title{A bijective proof for a theorem of Ehrhart}
\author{Steven V Sam}

\begin{document}

\maketitle

\begin{abstract} We give a new proof for a theorem of Ehrhart
  regarding the quasi-polynomiality of the function that counts the
  number of integer points in the integral dilates of a rational
  polytope. The proof involves a geometric bijection,
  inclusion-exclusion, and recurrence relations, and we also prove
  Ehrhart reciprocity using these methods.
\end{abstract}

\section{Introduction.}

Enumerative combinatorics is a rich and vast area of
study. Particularly interesting in this subject are families of
objects parameterized by the positive integers $\Z_{>0}$ with an
associated counting function $f(t)$ that is polynomial; this last
statement means that there is some polynomial $p(t)$ such that $p(t) =
f(t)$ for all $t \in \Z_{>0}$.

It is a bit mysterious that polynomial sequences arise at all in
enumerative combinatorics. Even more so, these polynomials should {\it
  a priori} have no meaning when evaluated at negative
values. However, the surprising fact is that they oftentimes do; such
occurrences are usually called {\bf combinatorial reciprocity
  theorems}. To warm up, we will begin with two examples of
combinatorial reciprocity theorems related to finite graphs and
partially ordered sets (posets, for short). The main subject of this
paper will be a combinatorial reciprocity theorem related to counting
lattice points in polytopes. To motivate this topic, we will discuss
Pick's theorem shortly, which is a special case of the reciprocity
theorem in dimension 2.

For the first example, let $\Gamma$ be a finite undirected graph with
$n$ vertices, and let $V(\Gamma)$ denote its vertex set. For $t \in
\Z_{>0}$, a {\bf $t$-coloring} of $\Gamma$ is a function $c \colon
V(\Gamma) \to \{1,\dots,t\}$. A $t$-coloring is {\bf proper} if $c(v)
\ne c(v')$ whenever $v$ and $v'$ are adjacent vertices. The function
$\chi_\Gamma(t)$ which counts the number of proper $t$-colorings of
$\Gamma$ is a polynomial of degree $n$ called the {\bf chromatic
  polynomial} of $\Gamma$. Surprisingly, there is a nice combinatorial
interpretation for the number $(-1)^n \chi_\Gamma(-t)$ for $t \in
\Z_{>0}$. First, some more definitions. Given an orientation of the
edges of $\Gamma$, a {\bf directed cycle} is a sequence of vertices
$(v_0, \dots, v_r)$ such that there is an edge oriented from $v_{i-1}$
to $v_i$ for $i=1,\dots, r$ and $v_0 = v_r$. An orientation of
$\Gamma$ is {\bf acyclic} if it has no directed cycles. We say that a
$t$-coloring is {\bf compatible} with an orientation of $\Gamma$ if
for every edge oriented from vertex $v_1$ to $v_2$, $c(v_1) \le
c(v_2)$. Then $(-1)^n \chi_\Gamma(-t)$ is the number of pairs
$(\alpha, c)$ where $\alpha$ is an acyclic orientation of $\Gamma$
and $c$ is $t$-coloring compatible with $\alpha$. By convention, if
$\Gamma$ has no edges, there is exactly one orientation on the edges
of $\Gamma$. In particular, $(-1)^n\chi_\Gamma(-1)$ counts the number
of acyclic orientations of $\Gamma$.

For the next example, let $P$ be a finite poset with $n$ elements.
The function $\Omega_P(t)$ which counts the number of order-preserving
maps $\phi \colon P \to \{1,\dots,t\}$, i.e., maps with the property
that if $x \le y$ for $x,y \in P$, then $\phi(x) \le \phi(y)$, is a
polynomial called the {\bf order polynomial} of $P$. The combinatorial
reciprocity theorem in the example of the order polynomial
$\Omega_P(t)$ is much simpler: $(-1)^n \Omega_P(-t)$ is the number of
{\it strict} order-preserving maps $\phi \colon P \to \{1,\dots,t\}$,
i.e., maps with the property that if $x < y$ for $x,y \in P$, then
$\phi(x) < \phi(y)$.

These interpretations are indeed a bit unexpected, but in the author's
opinion, this is one of the more attractive features of mathematics.

In this paper we will discuss another instance of combinatorial
reciprocity that is related to a famous theorem proven by Georg Pick.
Pick led a productive mathematical life and worked in many different
fields ranging from functional analysis and linear algebra to complex
analysis and differential geometry. His most famous result now is
Theorem \ref{picktheorem}, commonly known as Pick's theorem. When
first published, Pick's theorem did not receive much attention. It
was, however, included in the famous book {\it Mathematical Snapshots}
\cite{steinhaus} which was first published in 1969, and it then
attracted much more attention. During World War II, Pick was sent to
the Theresienstadt concentration camp in 1942 and died there shortly
after that. More information on Pick's life can be found at
\url{http://www-history.mcs.st-and.ac.uk/Biographies/Pick.html}.

Before stating Pick's theorem, we need a bit of notaton. Let $\P$ be a
connected and simply connected polygon (not necessarily convex, but we
do assume our polygons are simple) in the plane whose vertices lie in
$\Z^2$. For the rest of this paper, elements of $\Z^2$, and, more
generally, $\Z^n$, will be referred to as integer points, integral
points, and sometimes lattice points. Let $A$ be the area of $\P$, let
$B$ be the number of integer points on the boundary of $\P$, and let
$I$ be the number of integer points in the interior of $\P$. Pick's
famous theorem \cite{pick} (also see \cite[Theorem 2.8]{ccd} for a
modern treatment) relates these quantities:
\begin{theorem}[Pick] \label{picktheorem} Let $\P$ be a connected and
  simply connected\footnote{These hypotheses can be weakened. In the
    general case, we may allow $\P$ to have multiple connected
    components as long as each one has integral vertices, and we may
    also allow $\P$ to have ``holes'', as long as the ``vertices'' of
    the holes are also integer points. We won't bother with stating
    this precisely, but leave it to the reader to find the correct
    definitions. In this case, the $-1$ in Pick's theorem is replaced
    by $-\chi(\P)$, where $\chi$ denotes the Euler characteristic of
    $\P$ as a topological space (e.g., computed using singular
    homology). Recall that a contractible space (e.g., a connected and
    simply connected polygon) has Euler characteristic 1.} polygon in
  the plane whose vertices lie in $\Z^2$. With the notation above,
  \begin{align} \label{pick} 
    A = I + \frac{B}{2} - 1. 
  \end{align}
\end{theorem}

Now let $t$ be a positive integer; we consider dilates $t\P := \{tx
\mid x \in \P\}$. The area of $t\P$ is $At^2$ and the number of
integer points on the boundary of $\P$ is $Bt$. If we let $I(t)$
denote the number of integer points in the interior of $t\P$, then
\eqref{pick} becomes 
\begin{align} \label{dilatepick}
At^2 = I(t) + \frac{B}{2}t - 1.
\end{align}
We know that the number of integer points of $t\P$ is $I(t) + Bt$, so
by adding $\frac{B}{2}t + 1$ to both sides of \eqref{dilatepick}, we
find that the total number of integer points $L_\P(t)$ of $t\P$ is
\[
L_\P(t) = At^2 + \frac{B}{2}t + 1.
\]
The right-hand side is a quadratic polynomial in $t$. Let
$L_{\P^\circ}(t)$ denote the number of interior integer points of
$t\P$. The important observation is that
\[
L_{\P^\circ}(t) = At^2 - \frac{B}{2}t + 1,
\]
which leads to the functional equation
\begin{align} \label{pickequation}
L_{\P^\circ}(t) = L_\P(-t).
\end{align}
All of this can be illustrated by counting integer points in
Figure~\ref{pickfigure}.
\begin{figure}
  \begin{center}
    \includegraphics[width=.25\textwidth]{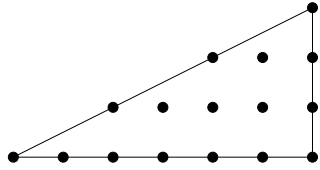}
  \end{center}
  \caption{An example of an integral polygon.}
  \label{pickfigure}
\end{figure}

This can, and will, be generalized to higher dimensions. However, a
bit should be said about how one might generalize Pick's theorem. For
example, can one compute the volume of a 3-dimensonal polyhedron by
counting its integer points? The answer is no, and comes in the form
of an example:

\begin{example} \label{reeveexample} Let $T_h$ be the tetrahedron
  whose vertices are the points $(0,0,0)$, $(1,0,0)$, $(0,1,0)$, and
  $(1,1,h)$ for $h \in \Z_{>0}$; its base is the triangle whose
  vertices are the first three points mentioned. The volume of $T_h$
  is $\frac{1}{3}$ times the area of the base times the height, which
  is $h$. This comes out to $\frac{h}{6}$. It is not hard to see that
  the only integer points inside of $T_h$ are the four points
  mentioned above: a general point $P \in T_h$ looks like
  \[
  P = a(0,0,0) + b(1,0,0) + c(0,1,0) + d(1,1,h)
  \]
  where $a,b,c,d \ge 0$ and $a + b + c + d=1$. Then if $P \in \Z^3$,
  we have to have $b,c,d \in \{0,1\}$. But if any of them is 1, then
  the other three coefficients must be 0, and if they are all 0, then
  $P$ is the origin. So we are left with the fact that $T_h$ {\it
    always} has 4 integer points but the volume grows arbitrarily
  large as $h$ tends to infinity. Thus, no higher-dimensional analogue
  of Pick's theorem can hold.

  It is worth mentioning that this example was first used by John
  Reeve in 1957 (see \cite{reeve}) to show that the idea of computing
  area from counting integer points does not generalize to 3
  dimensions. For the reader who comes back to this example, its
  Ehrhart polynomial is
  \[
  L_{T_h}(t) = \frac{h}{6}t^3 + t^2 + \left(2-\frac{h}{6}\right)t + 1,
  \]
  so $T_{13}$ has negative coefficients in its Ehrhart polynomial.
\end{example}

Nevertheless, there is a theorem (Theorem \ref{ehrharttheorem}), now
called Ehrhart's theorem, which is, in a vague sense, the correct way
to generalize Pick's theorem to higher dimensions. It was proven by
Eug\`ene Ehrhart, who was not a professional mathematician. He spent
most of his life teaching mathematics in high schools in France and
did his research on the side as a hobby. He proved his eponymous
theorem in 1962 \cite{ehrhartpolynomial}, and it wasn't until the age
of 60 that he obtained his Ph.D. with his thesis, {\it Sur un
  probl\`eme de g\'eom\'etrie diophantienne lin\'eaire} ({\it On a
  linear problem in Diophantine geometry}). Most of his papers concern
discrete geometry and Diophantine equations. For more information
about Ehrhart, the reader might see the website
\url{http://icps.u-strasbg.fr/~clauss/Ehrhart.html}.

In this paper, we are interested in the Ehrhart polynomial of an
integral polytope. In Section~\ref{quasipolysection}, we shall be
interested in the Ehrhart {\it quasi}-polynomial of a rational
polytope. Though we haven't defined these terms yet, what's to come
should be clear: we will construct a counting function associated to
an integral polytope, show that it agrees with a polynomial for
positive integers (in fact also for 0), and then derive a
combinatorial reciprocity theorem. While the theorems are originally
due to Ehrhart \cite{ehrhartpolynomial} and Macdonald
\cite{macdonald}, our proof is new. In particular, our proof of
Ehrhart--Macdonald reciprocity clears up some of the mystery (see
Figure~\ref{reciprocityfig}) of its statement. While the conceptual
idea of the proof is simple, verifying the details involves many
manipulations of summations, which can be a bit exhausting. To remedy
this, we have provided general ideas of how the proofs are to work
through examples before each proof.

Before proceeding, we should mention that the first two examples
presented in the introduction are special cases of Ehrhart's theorem,
because one can translate problems about counting proper colorings or
order preserving maps into counting integer points in some integral
polytope (or at least something approximately equal to an integral
polytope for which Ehrhart's theorem is true). For the connection
between chromatic polynomials (and more) and counting integer points,
the reader is encouraged to read \cite{insideout}, and for the
connection with order polynomials, the article \cite{posetpolytope} is
recommended.

\section{Statements of results.} \label{statements}

We will now give some definitions and explain the general setup. Given
points $p_1, \dots, p_n \in \R^n$, a {\bf convex combination} of $p_1,
\dots, p_n$ is a linear combination $a_1v_1 + \cdots + a_nv_n$ where
$a_i \ge 0$ for $i=1,\dots,n$ and $a_1 + \cdots + a_n = 1$. The {\bf
  convex hull} of a set $S \subseteq \R^n$ is the set of all convex
combinations of $S$:
\[
\conv(S) := \{ a_1x_1 + \cdots + a_rx_r \mid a_i \ge 0,\ a_1 + \cdots
+ a_r = 1,\ x_i \in S\}.
\]
Another equivalent definition is that the convex hull of $S$ is the
intersection of all convex sets containing $S$; we'll be more
interested in the first definition. An {\bf integral} (respectively,
{\bf rational}) {\bf polytope} $\P \subset \R^n$ is the convex hull of
finitely many integral (respectively, rational) points in $\R^n$. The
{\bf dimension} of $\P$ is the dimension of its affine span (or,
equivalently, one could translate some point in $\P$ to the origin and
compute the dimension of the vector subspace of $\R^n$ that its points
generate). If $v \in \P$ cannot be written as a convex combination of
any subset of points in $\P$ that does not include $v$, then $v$ is a
{\bf vertex} of $\P$. In particular, $\P$ is the convex hull of its
vertices, and there are only finitely many of them. If the dimension
of $\P$ is $d$, and $\P$ has $d+1$ vertices, we say that $\P$ is a
{\bf simplex}. Note that by linear independence, the representation of
a point inside of a simplex as a convex combination of its vertices is
necessarily unique. Let $\P^{\circ}$ denote the {\bf relative
  interior} of $\P$, i.e., the topological interior of $\P$ in its
affine span with the subspace topology. It is not hard to see that the
relative interior of a simplex with vertices $v_0, \dots, v_d$ is the
set of convex combinations $a_0v_0 + \cdots + a_dv_d$ where $a_i > 0$
for $i=0,\dots,d$.

Given a polytope $\P$, we define a scalar multiplication $t\P := \{ tx
\mid x \in \P \}$ for $t \in \R$, but we shall restrict our attention
to $t \in \Z$. Now define $\ell_\P \colon \Z \to \Z_{\ge 0}$ by
\[
\ell_\P(t) := \begin{cases} \#(t\P \cap \Z^n) & \text{if } t > 0,\\
  1 & \text{if } t=0,\\
  (-1)^{\dim \P}\#(t\P^\circ \cap \Z^n) & \text{if } t < 0. \end{cases}
\]
Here $\#$ denotes the cardinality of a set. This
definition\footnote{The reader may have noticed that for the purposes
  of counting integer points, it makes no difference if we consider
  $t\P^\circ \cap \Z^n$ or $-t\P^\circ \cap \Z^n$ when $t<0$, but it
  will turn out in the proof of Theorem~\ref{reciprocity} that
  $t\P^\circ \cap \Z^n$ is the ``correct'' definition. Furthermore,
  there should be no reason to separate the case $t=0$ because $0\P$
  is a single integer point at the origin. In this case it is
  irrelevant because polytopes are contractible, but for the case of
  polytopal complexes (which we can still count!), $t=0$ indeed
  becomes an exceptional case.} may seem strange, but now the goal of
this paper becomes easy to state:

\begin{theorem}[Ehrhart, Macdonald] If $\P \subset \R^n$ is an
  integral polytope of dimension $d$, then there exists a polynomial
  $L_\P(t)$ of degree $d$ such that $L_\P(t) = \ell_\P(t)$ for all $t
  \in \Z$.
\end{theorem}

Unfolding this compact statement, we obtain the following two
theorems.

\begin{theorem}[Ehrhart] \label{ehrharttheorem} If $\P \subset \R^n$
  is an integral polytope of dimension $d$, then the function $\#(t\P
  \cap \Z^n)$ agrees with a polynomial $L_\P(t)$ of degree $d$ for all
  nonnegative integers.
\end{theorem}

The polynomial $L_\P(t)$ is called the {\bf Ehrhart polynomial} of
$\P$. The combinatorial reciprocity theorem associated with it is the
following statement.

\begin{theorem}[Ehrhart--Macdonald reciprocity] \label{reciprocity} If
  $\P \subset \R^n$ is an integral polytope of dimension $d$, then for
  $t \in \Z_{>0}$,
  \[
  (-1)^d L_{\P}(-t) = \#(t\P^\circ \cap \Z^n).
  \]
\end{theorem}

Compare this with \eqref{pickequation}. Our proof of these theorems
uses the following standard result \cite[Corollary 4.3.1]{stanleyec1}.

\begin{lemma} \label{recurrencelemma} For $f \colon \Z_{\ge 0} \to \C$
  and $d \in \Z_{\ge 0}$, the following are equivalent:
  \begin{enumerate}[(i)]
  \item \label{rationalitem} There exists $P(z) \in \C[z]$ with $\deg
    P \le d$ such that
    \[ \displaystyle \sum_{t \ge 0} f(t) z^t =
    \frac{P(z)}{(1-z)^{d+1}}\,. \]
  \item \label{recurrenceitem} For all $t \ge 0$,
    \[ 
    \sum_{k=0}^{d+1}(-1)^{d+1-k} \binom{d+1}{k} f(t+k) = 0.
    \] 
  \item \label{polynomialitem} There is a polynomial of degree $\le d$
    that agrees with $f(t)$ for all nonnegative integers.
  \end{enumerate}
\end{lemma}

The original proof of Theorem~\ref{ehrharttheorem} by Ehrhart
\cite{ehrhartpolynomial} uses the equivalence of items
(\ref{rationalitem}) and (\ref{polynomialitem}) from Lemma
\ref{recurrencelemma}, but we shall make use of the equivalence of
items (\ref{recurrenceitem}) and (\ref{polynomialitem}). For another
account of Ehrhart's proof, the book \cite{ccd} is recommended. A
completely different approach using the machinery of toric varieties
can be found in \cite[Chapter 13]{mustata}. We should mention that
though the algebro-geometric proof of Ehrhart's theorem uses much more
machinery than may seem necessary, it does have the nice feature that
it does not appeal to triangulations to reduce to the case that the
polytope is a simplex. Going back to our proof, since the equivalence
of items (\ref{recurrenceitem}) and (\ref{polynomialitem}) is so
crucial to our approach, we will give a proof.

\begin{proof}[Proof of equivalence of items (\ref{recurrenceitem}) and
  (\ref{polynomialitem}) of Lemma \ref{recurrencelemma}]

  The proof is by induction on $d$. If $d=0$, then the equivalence
  says that $f$ is a constant function if and only if $f(t+1) - f(t) =
  0$ for all $t$, which is clear. Now suppose that $d>0$, and assume
  that the equivalence holds for $d-1$. 

  Suppose that (\ref{polynomialitem}) holds. Then $g(t) = f(t+1) -
  f(t)$ is a polynomial of degree $\le d-1$, so by the inductive
  hypothesis,
  \begin{align*}
    0 &= \sum_{k=0}^d (-1)^{d-k} \binom{d}{k} (f(t+1+k) - f(t+k))\\
    &= \sum_{k=1}^{d+1} (-1)^{d-k+1} \binom{d}{k-1} f(t+k) +
    \sum_{k=0}^d (-1)^{d-k+1} \binom{d}{k} f(t+k)\\
    &= f(t+d+1) + (-1)^{d+1} f(t) + \sum_{k=1}^d (-1)^{d-k+1}
    \binom{d+1}{k} f(t+k)\\
    &= \sum_{k=0}^{d+1} (-1)^{d-k+1} \binom{d+1}{k} f(t+k).
  \end{align*}

  Now suppose that (\ref{recurrenceitem}) holds. Running backwards
  through the above calculations, we see that the function $g(t) =
  f(t+1) - f(t)$ satisfies (\ref{recurrenceitem}) for $d-1$ instead of
  $d$, so by the inductive hypothesis, there is a polynomial of degree
  $\le d-1$ that agrees with $g(t)$ for all nonnegative integers. Then
  we can write $f(t+1) = g(t) + f(t)$, and by induction on $t$ this
  becomes
  \[
  f(t+1) = f(0) + \sum_{k=0}^t g(k).
  \]
  So it is enough to check that the sum on the right is a
  polynomial. By breaking $g$ up into monomials, we can reduce to
  showing that $\sum_{k=0}^t k^r$ is a polynomial. This is a
  well-known fact, but here is a short proof. Since $t^r$ is a
  rational polynomial, it is in the rational vector space with basis
  $\binom{t}{i}$ for $i \in \N$. So we can make the further reduction
  of showing that $\sum_{k=0}^t \binom{k}{i}$ is a polynomial for
  fixed $i$. But this is true because of the identity
  \[
  \sum_{k=0}^t \binom{k}{i} = \binom{t+1}{i+1}.
  \]
  To see why this identity holds, note that the right-hand side counts
  $(i+1)$-subsets of $\{1, \dots, t+1\}$, while the left-hand side
  counts the same thing if we interpret each $\binom{k}{i}$ as
  counting the number of $(i+1)$-subsets of $\{1, \dots, k+1\}$ which
  contain $k+1$.
\end{proof}

\section{The Ehrhart polynomial of an integral polytope.}

To get a feel for the geometric idea behind the proof of Theorem
\ref{ehrharttheorem}, we begin by considering the polytope $\P$ whose
vertices are $(0,0)$, $(2,0)$, and $(2,1)$. The large triangle in
Figures~\ref{coveringfig} and \ref{intersectionfig} is $3\P$, and the
shaded subtriangles display the following recurrence relation:
\[
\ell_\P(3) = 3\ell_\P(2) - 3\ell_\P(1) + \ell_\P(0).
\]

\begin{proof}[Proof of Theorem~\ref{ehrharttheorem}] We will show that
  \begin{align} \label{incexcl} \ell_{\P}(t+d+1) = \sum_{k=0}^d
    (-1)^{d-k} \binom{d+1}{k} \ell_{\P}(t+k)
  \end{align}
  for all $t \ge 0$; then Lemma~\ref{recurrencelemma} gives the
  polynomiality of the sequence $\ell_{\P}(t)$. It is sufficient to
  prove \eqref{incexcl} for simplices because any integral polytope
  $\P$ can be triangulated\footnote{Our definition of a triangulation
    of a polytope $\P$ is a finite collection of simplices $\T =
    \{T_i\}$ such that (1) $\bigcup T_i = \P$, (2) if $T'$ is a face
    of some $T_i \in \T$, then $T' \in \T$, and (3) for $T_i, T_j \in
    \T$, the intersection of $T_i$ and $T_j$ is a face of both $T_i$
    and $T_j$. } into simplices $\{T_i\}$ such that each vertex of
  each $T_i$ is a vertex of $\P$ (a proof of this can be found in
  \cite[Appendix B]{ccd}). Inclusion-exclusion then gives
  $\ell_{\P}(t)$ as a sum of the $\ell_{T_i}(t)$ with appropriate
  signs. So without loss of generality, we may assume that $\P$ is a
  simplex.

  Let $\{v_0, \dots, v_d\}$ be the vertices of $\P$ and fix an integer
  $t \ge 0$. For each vertex $v_i$ of $\P$, define $Q_i := (t+d)\P +
  v_i$. See Figure~\ref{coveringfig} for an example where $d=2$,
  $t=0$, and $\P$ is the convex hull of $\{(0,0), (2,0), (2,1)\}$.
  \begin{figure}
    \begin{center}
      \includegraphics[width=.75\textwidth]{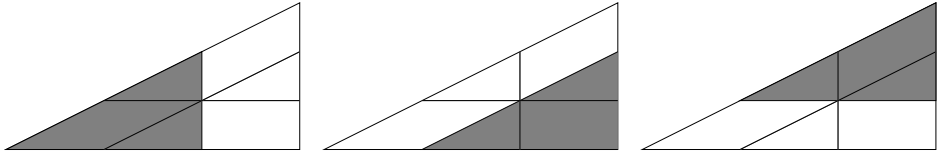}
    \end{center}
    \caption{From left to right: $Q_0, Q_1, Q_2$}
    \label{coveringfig}
  \end{figure}
  We use inclusion-exclusion to compute the number of integer points
  in $Q := \bigcup_i Q_i$. That is, we add the number of integer
  points that are contained in each $Q_i$, subtract those that are
  contained in each intersection of two $Q_i$, etc. By our
  construction of these simplices, we can describe the $k$-fold
  intersections explicitly. For our running example, see
  Figure~\ref{intersectionfig}.
  \begin{figure}
    \begin{center}
      \includegraphics[width=.75\textwidth]{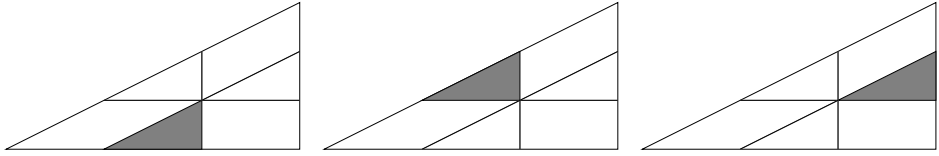}
    \end{center}
    \caption{From left to right: $Q_0 \cap Q_1, Q_0 \cap Q_2, Q_1 \cap
      Q_2$.}
    \label{intersectionfig}
  \end{figure}
  The first observation is that
  \begin{align*}
    Q_j &= (t+d)\P + v_j\\
    &= \left\{ v_j + \sum_{i=0}^d (t+d)c_iv_i \ \middle|\ c_i \ge 0,
      \sum_{i=0}^d c_i = 1 \right\}\\
    &= \left\{ \sum_{i=0}^d a_iv_i \ \middle|\ a_j \ge 1, a_i \ge 0
      \text{ if } i \ne j, \text{ and } \sum_{i=0}^d a_i = t+d+1
    \right\},
  \end{align*}
  so for any $I \subseteq D := \{0,\dots,d\}$,
  \begin{align*}
    \bigcap_{i \in I} Q_i &= \left\{ \sum_{i=0}^d a_iv_i \ \middle|\
      a_i \ge 1 \text{ if } i \in I, a_i \ge 0 \text{ if } i \notin I,
      \text{ and } \sum_{i=0}^d a_i = t+d+1 \right\}\\
    &= (t + d + 1 - \#I)\P + \sum_{i \in I} v_i.
  \end{align*}
  For each $k=1,\dots,d+1$, there are $\binom{d+1}{d+1-k}$ $k$-fold
  intersections, and each contains $\ell_\P(t+d+1-k)$ integer points
  because the sets differ from one another by an integer translate. So
  inclusion-exclusion gives
  \begin{align*}
    \#(Q \cap \Z^n) &= \sum_{k=1}^{d+1} (-1)^{k+1} \sum_{\substack{I
        \subseteq D\\ \#I = k}} \#\left(\bigcap_{i \in I} Q_i \cap
      \Z^n\right)\\
    &= \sum_{k=1}^{d+1} (-1)^{k+1} \binom{d+1}{d+1-k} \ell_{\P}(t+d+1-k)\\
    &= \sum_{k=0}^d (-1)^{d-k} \binom{d+1}{k} \ell_{\P}(t+k).
  \end{align*}
  Note that if $t = k = 0$, then our definition $\ell_\P(0) = 1$
  coincides with the fact that the intersection of all the $Q_j$ is a
  single integer point. The right-hand side of this equation coincides
  with the right-hand side of \eqref{incexcl}. To finish, we show that
  $Q = (t+d+1)\P$. It is clear that $Q \subseteq (t+d+1)\P$. To prove
  the other inclusion, first note that
  \[
  (t+d+1)\P = \left\{ \sum_{i=0}^d a_iv_i \ \middle|\ a_i \ge 0,\
    \sum_{i=0}^d a_i = t+d+1 \right\}.
  \]
  Since $t \ge 0$, it follows that for any point $P = a_0v_0 + \cdots
  + a_dv_d \in (t+d+1)\P$, there must exist some $j$ such that $a_j
  \ge 1$. (This is one of the reasons why it is important that we are
  assuming that $\P$ is a simplex.) Then $P \in Q_j$, so $(t+d+1)\P
  \subseteq Q$, and we conclude that there exists a polynomial
  $L_\P(t)$ such that $L_\P(t) = \#(t\P \cap \Z^n)$ for $t \in \Z_{\ge
    0}$.

  Finally, we must show that the degree of the polynomial $L_\P(t)$ is
  $d$. The above work shows that $L_\P(t)$ is a polynomial of degree
  at most $d$. By translating $\P$ if necessary, we may assume that
  one of its vertices is the origin. Since $\P$ is $d$-dimensional,
  there are vertices $v_1, \dots, v_d$ that are linearly independent
  when considered as vectors. For positive integers $k_1, \dots, k_d
  \le t$, the point $k_1v_1 + \cdots + k_dv_d$ lives in $dt\P \cap
  \Z^n$, and these points are all distinct for different choices of
  $k_i$ by linear independence of the $v_i$. Hence $L_\P(dt) \ge t^d$,
  which shows that $L_\P(t)$ has degree at least $d$.
\end{proof}

\section{Ehrhart--Macdonald reciprocity.}

As in the case of the proof of Theorem~\ref{ehrharttheorem}, the idea
behind the proof of Theorem~\ref{reciprocity} can be seen in Figure
\ref{reciprocityfig}: again we are considering the polytope $\P$ whose
vertices are $(0,0)$, $(2,0)$, and $(2,1)$. The figure depicts $2\P$
and illustrates the recurrence relation
\[
\ell_\P(2) = 3\ell_\P(1) - 3\ell_\P(0) + \ell_\P(-1).
\]

\begin{proof}[Proof of Theorem~\ref{reciprocity}] We first handle the
  case when $\P$ is a simplex. Going back to the proof of
  Lemma~\ref{recurrencelemma}, it is clear that if $f \colon \Z_{\ge
    N} \to \C$ for some integer $N$ then the statement that item
  (\ref{recurrenceitem}) of Lemma~\ref{recurrencelemma} holds for all
  $t \ge N$ is equivalent to the statement that there exists a
  polynomial of degree $\le d$ that agrees with $f(t)$ for all $t \ge
  N$. So to prove that there is a polynomial that agrees with
  $\ell_\P(t)$ for {\it all} integers, it will be enough to show that
  \eqref{incexcl} holds for all integers $t$.

  The content of Theorem~\ref{ehrharttheorem} is the case $t \ge
  0$. For $t \le -d-1$, the proof is similar to the proof for
  Theorem~\ref{ehrharttheorem} because every occurrence of $\P$ can be
  replaced by $-\P^\circ$, and the statements are valid after
  replacing weak inequalities (inside of the set descriptions of the
  $Q_i$'s and their intersections) with strict inequalities. So we may
  assume that $0 > t > -d-1$. As before, define $Q_i := (t+d)\P + v_i$
  and $Q := \bigcup_i Q_i$. Then the equality
  \begin{align} \label{qpoints}
  \#(Q \cap \Z^n) = \sum_{k=1}^{t+d+1} (-1)^{k+1} \binom{d+1}{d+1-k}
  \ell_{\P}(t+d+1-k)
  \end{align}
  holds by the same reasoning as in the proof of
  Theorem~\ref{ehrharttheorem}. However, we cannot say that $(t+d+1)\P
  = Q$. Indeed, we can describe this deficiency explicitly:
  \[
  (t+d+1)\P \setminus Q = \left\{ \sum_{i=0}^d a_iv_i \ \middle|\ 0
    \le a_i < 1,\ \sum_{i=0}^d a_i = t+d+1 \right\}.
  \]
  See Figure~\ref{reciprocityfig} for an example in which $d=2$,
  $t=-1$, and $\P$ is the convex hull of $\{(0,0), (2,0), (2,1)\}$. In
  this example, note that the hole is precisely $-\P^\circ + (4,2)$.
  \begin{figure}
    \begin{center}
      \includegraphics[width=.25\textwidth]{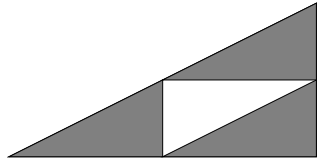}
    \end{center}
    \caption{The deficiency $(t+d+1)\P \setminus Q$.}
    \label{reciprocityfig}
  \end{figure}
  Now define 
  \[
  \P' := t\P^\circ + \sum_i v_i.
  \]
  First note that
  \[
  t\P^\circ = \left\{ \sum_{i=0}^d a_iv_i \ \middle|\ a_i < 0,\
    \sum_{i=0}^d a_i = t \right\},
  \]
  which implies
  \[
  \P' = t\P^\circ + \sum_{i=0}^d v_i = \left\{ \sum_{i=0}^d a_iv_i \
    \middle|\ a_i < 1,\ \sum_{i=0}^d a_i = t+d+1 \right\}.
  \]
  If $t = -1$, then $\P' = (t+d+1)\P \setminus Q$ because each
  coefficient $a_i$ needs to be nonnegative if they are to sum to
  $d$. Otherwise, we can try to cover $\P'$ by simplices of the form
  \[
  Q'_j := (t+1)\P^\circ + \sum_{\substack{i=0\\i \ne j}}^d v_i
  \]
  as in Theorem~\ref{ehrharttheorem}. Define $Q' := \bigcup_i Q'_i$
  for $t < -1$ and $Q' = \varnothing$ for $t=-1$. We shall show that
  $\P' \setminus Q' = (t+d+1)\P \setminus Q$. The case $t=-1$ was
  discussed above, so assume $t<-1$. Then
  \begin{align*}
    Q'_j &= (t+1)\P^\circ + \sum_{\substack{i=0\\i \ne j}}^d v_i\\
    &= \left\{ \sum_{i=0}^d a_iv_i \ \middle|\ a_j < 0, a_i < 1 \text{
        if } i \ne j, \text{ and } \sum_{i=0}^d a_i = t+d+1 \right\},
  \end{align*}
  so
  \begin{align*}
    \P' \setminus Q' = \left\{ \sum_{i=0}^d a_iv_i \ \middle|\ 0 \le
      a_i < 1,\ \sum_{i=0}^d a_i = t+d+1 \right\} = (t+d+1)\P
    \setminus Q.
  \end{align*}
  Inclusion-exclusion once again gives (remember what $\ell_\P(t)$
  means when $t$ is negative!)
  \begin{align*}
    \#(Q' \cap \Z^n) &= \sum_{k=1}^{-(t+1)} (-1)^{k+1}
    \binom{d+1}{d+1-k} \#((t+k)\P^\circ \cap \Z^n) \\
    &= \sum_{k=1}^{-(t+1)} (-1)^{k+1} \binom{d+1}{d+1-k}
    (-1)^d\ell_\P(t+k).
  \end{align*}
  This holds even for $t=-1$ because the sum on the right-hand side is
  empty in this case. This implies that
  \begin{align} \label{p'minusq'points} 
    \#((\P' \setminus Q') \cap
    \Z^n) &= \#(\P' \cap \Z^n) - \#(Q' \cap \Z^n)\notag \\
    &= \sum_{k=0}^{-(t+1)} (-1)^{k+d} \binom{d+1}{d+1-k} \ell_\P(t+k).
  \end{align}

  Finally, since we know that $(t+d+1)\P \setminus Q = \P' \setminus
  Q'$, we can write $(t+d+1)\P$ as the disjoint union of $Q$ and $\P'
  \setminus Q'$. Therefore, combining \eqref{qpoints} and
  \eqref{p'minusq'points},
  \begin{align*}
    \ell_\P(t+d+1) &= \#((t+d+1)\P \cap \Z^n)\\
    &= \#(Q \cap \Z^n) + \#((\P' \setminus Q') \cap \Z^n)\\
    &= \sum_{k=1}^{t+d+1} (-1)^{k+1} \binom{d+1}{d+1-k}
    \ell_\P(t+d+1-k) + \sum_{k=0}^{-(t+1)} (-1)^{k+d}
    \binom{d+1}{d+1-k} \ell_\P(t+k)\\
    &= \sum_{k=-t}^d (-1)^{d-k} \binom{d+1}{k} \ell_\P(t+k) +
    \sum_{k=0}^{-(t+1)} (-1)^{k+d} \binom{d+1}{k} \ell_\P(t+k)\\
    &= \sum_{k=0}^d (-1)^{d-k} \binom{d+1}{k} \ell_\P(t+k),
  \end{align*}
  which finishes the proof for simplices.

  For the general case, let $\P$ be an integral polytope with more
  than $d+1$ vertices. Triangulate $\P$ using only integral vertices;
  call this triangulation $\T$. There is a natural poset structure on
  $\T$, namely $T_i \le T_j$ if $T_i$ is a face of $T_j$. Let $L(\P)$
  denote this poset with an additional element $\hat{1}$ such that
  $\hat{1} \ge T_i$ for all $T_i \in \T$. We finish the proof for $\P$
  via the M\"obius inversion formula \cite[Proposition
  3.7.1]{stanleyec1} on $L(\P)$. Fix some $t \in \Z_{>0}$. Define $f
  \colon L(\P) \to \Z_{\ge 0}$ by $f(\F) = L_\F(t)$ where $\F$ is a
  face of $\T$ and $f(\hat{1}) = L_\P(t)$. Also, define $g \colon
  L(\P) \to \Z_{\ge 0}$ by $g(\F) = (-1)^{\dim \F}L_\F(-t)$ where $\F$
  is a face of $\T$ and $g(\hat{1}) = 0$. Because every point of $\P$
  lies in the relative interior of a unique face of $\T$, we know that
  \[
  f(\hat{1}) = \ell_\P(t) = \sum_{\F < \hat{1}} \#(t\F^\circ \cap
  \Z^n) = \sum_{\F \le \hat{1}} g(\F),
  \]
  and by the M\"obius inversion formula, this is equivalent to
  \begin{align} \label{mobius}
  0 = g(\hat{1}) = \sum_{\F \le \hat{1}} \mu(\F, \hat{1}) f(\F),
  \end{align}
  where $\mu$ is the M\"obius function on $L(\P)$. Appealing to
  \cite[Proposition 3.8.9]{stanleyec1},
  \[
  \mu(\F,\hat{1}) = \begin{cases} 0 & \text{if } \F \subseteq \partial
    \P \text{ or } \F = \varnothing,\\
    1 & \text{if } \F = \hat{1},\\
    (-1)^{d-\dim \F+1} & \text{otherwise.}
\end{cases}
  \]
  Now \eqref{mobius} becomes
  \[
  0 = f(\hat{1}) + \sum_{\F \in \T^\circ} (-1)^{d-\dim \F +1} f(\F),
  \]
  where $\T^\circ$ is the set of faces of $\T$ that do not lie on the
  boundary of $\P$. In a nicer form, this is
  \[
  L_\P(t) = (-1)^d \sum_{\F \in \T^\circ} (-1)^{\dim \F} L_\F(t).
  \]
  The functions involved are polynomials, so since they agree at all
  positive integers, they are equal as functions. The last step is to
  evaluate at $-t$:
  \begin{align*}
    L_\P(-t) &= (-1)^d \sum_{\F \in \T^\circ} (-1)^{\dim \F} L_\F(-t)\\
    &= (-1)^d \sum_{\F \in \T^\circ} \#(t\F^\circ \cap \Z^n)\\
    &= (-1)^d \#(t\P^\circ \cap \Z^n). \qedhere
  \end{align*}
\end{proof}

\section{The Ehrhart quasi-polynomial of a rational
  polytope.} \label{quasipolysection} 

Now that we have obtained our objective, we generalize to rational
polytopes. To do so, we need some more definitions. The {\bf
  denominator} of a rational polytope $\P$ is the smallest positive
integer $D$ such that $D\P$ is an integral polytope.

A {\bf quasi-polynomial} $p$ with period $s$ is a piecewise defined
function
\[
p(t) = p_i(t)\quad \text{if }t \equiv i \pmod s,
\]
where the $p_i$ are polynomials. The {\bf degree} of $p$ is the
largest degree of the $p_i$. Equivalently, a quasi-polynomial is a
polynomial whose coefficients are periodic functions with finite
period.

\begin{corollary}[Ehrhart--Macdonald] Let $\P \subset \R^n$ be a
  rational polytope of dimension $d$ with denominator $s$. Then
  $L_{\P}(t)$ is a quasi-polynomial of degree $d$ with period dividing
  $s$, and
  \[
  (-1)^d L_{\P}(-t) = \#(t\P^\circ \cap \Z^n)
  \]
  for all $t \in \Z_{>0}$.
\end{corollary}

\begin{proof} Again assume $\P$ is a simplex. The only place that
  integrality was required in the proof of
  Theorem~\ref{ehrharttheorem} is in describing the $k$-fold
  intersections of the $Q_i$. That is, we translated certain sets by
  integral points to get the correct set-theoretic arguments. We can
  do the same thing now, except that now one translates by $sv_i$
  where $v_i$ is a vertex to guarantee preservation of lattice
  points. Thus, for each $0 \le j < s$, the sequence
  $(L_{\P}(ts+j))_{t \in \Z}$ satisfies the condition for
  polynomiality. The jump from simplices to polytopes is the same as
  before.
\end{proof}

\section{Concluding remarks.}

Recalling the example in the introduction on Pick's theorem, there
were interpretations for the coefficients of $L_\P(t)$ when $d=2$. A
more careful estimate of $L_\P(t)$ in the proof of
Theorem~\ref{ehrharttheorem} would show that for general $d$,
$L_\P(t)$ is asymptotic to $t^d\vol(\P)$, where $\vol(\P)$ denotes the
{\bf relative volume} of $\P$, which is the volume of $\P$ relative to
the lattice of its affine span. Thus, the leading coefficient of
$L_\P(t)$ is $\vol(\P)$. The fact that the constant coefficient is 1
follows from the fact that the Euler characteristic of a polytope is
1, and that Euler characteristic is additive with respect to
inclusion-exclusion. To understand the second leading coefficient
$c_{d-1}$ of $L_\P(t)$, we can use Ehrhart--Macdonald reciprocity to
conclude that
\[
\#(\partial \P \cap \Z^n) = L_\P(1) - (-1)^dL_\P(-1),
\]
and the leading coefficient of the right-hand side is $2c_{d-1}$. This
means that $2c_{d-1}$ is the sum of the relative volumes of the facets
of $\P$. With just the results in this paper, this is where we must
stop. Even worse, the coefficients of Ehrhart polynomials may be
negative in some cases (see Example \ref{reeveexample}), so it is not
even clear what to guess the other coefficients might be telling us.

If we allow ourselves to pass to the world of algebraic geometry, then
the coefficients of the Ehrhart polynomial can be expressed via
intersections of Todd classes on the associated toric variety of
$\P$. For more details, the reader is referred to \cite[Section
5.3]{fulton}. In general, however, these intersection numbers are
quite difficult to compute. But with some hard work, one can
understand the linear coefficient for $d = 3$ in terms of Dedekind
sums; this is done in \cite{pommersheim}. This work was generalized in
\cite[Corollary 1]{diaz}, which allows one to obtain the coefficients
of the other terms via Fourier analysis.

~

In general, it is difficult to determine the minimum period of
$L_\P(t)$. Indeed, there even exist examples of nonintegral polytopes
whose Ehrhart quasi-polynomial has period 1. The article
\cite{minperiod} constructs examples for all dimensions $\ge 2$ and
for arbitrary denominator. For more information, the article
\cite{maximalperiod} constructs simplices whose Ehrhart
quasi-polynomial has coefficient functions with prescribed minimum
periods, and the article \cite{periodcollapse} offers some conjectures
for why the minimum period of $L_\P(t)$ is sometimes strictly smaller
than the denominator of $\P$.

~

Consider the following generalization of counting integer points in
$\P$. Instead of counting each point as 1, we weight the points by
their solid angles. Given a polytope $\P \subset \R^n$ and a point $x
\in \R^n$, define the {\bf solid angle} at $x$ with respect to $\P$ to
be
\[
\omega_\P(x) := \lim_{r \to 0} \frac{\vol (B_r(x) \cap \P)}{\vol
  B_r(x)},
\]
where $B_r(x)$ denotes the ball of radius $r$ centered at $x$, and
$\vol$ denotes the usual Euclidean volume in $\R^n$. We should assume
$\P$ is $n$-dimensional, otherwise this limit is always 0, which is
quite boring. This ratio is eventually constant for sufficiently small
$r$, so $\omega_\P(x)$ is well-defined, and we can instead ask about
the solid angle enumerator
\[
a_\P(t) := \sum_{x \in \Z^n} \omega_{t\P}(x).
\]
Note that the sum on the right is actually finite because for $x
\notin t\P$, $\omega_{t\P}(x) = 0$. Going through the proof of
Theorem~\ref{ehrharttheorem}, it is immediate that it generalizes to
the sequence for $\{a_\P(t)\}_{t \in \Z_{>0}}$, so there is a
polynomial $A_\P(t)$ that agrees with $a_\P(t)$ for all $t \in
\Z_{>0}$. We call $A_\P(t)$ the {\bf solid angle polynomial} of
$\P$. By the way, the right way to extend this sequence to $\Z_{\le
  0}$ can be seen from a careful analysis of
Figure~\ref{reciprocityfig}: define
\[
a_\P(-t) := (-1)^n \sum_{x \in \Z^n} \omega_{-t\P}(x)
\]
for $t \in \Z_{>0}$ and $a_\P(0) := 0$. We do not take $-t\P^\circ$
because if two simplices $\Delta_1$ and $\Delta_2$ meet in a facet of
both, and we pick $x \in \Delta_1 \cap \Delta_2$, then
\[
\omega_{\Delta_1 \cup \Delta_2}(x) = \omega_{\Delta_1}(x) +
\omega_{\Delta_2}(x).
\]
In other words, inclusion-exclusion is easy for solid angles because
there are no overlaps! Reciprocity for solid angles tells us simply
that $A_\P(t)$ is either an even or odd function depending on the
parity of $n$. Of course, all of the above discussion can be extended
to rational polytopes by replacing polynomials with
quasi-polynomials. The theory of solid angles of polytopes is still
poorly understood, and the reader is referred to \cite[Chapter
11]{ccd} for some open problems. The recent paper \cite{desario}
extends the theory of solid angles in rational polytopes and integral
dilates to solid angles in arbitrary real polytopes and real dilates
using techniques from harmonic analysis.

\paragraph{Acknowledgements.} The author thanks Aaron Dall for helpful
discussions. The author also thanks Allen Knutson, Richard Stanley,
and Robin Chapman for pointing out corrections and improvements to a
previous draft. Finally the author thanks Matthias Beck for help with
organizational issues of this paper and two anonymous referees for
helpful suggestions.


\bigskip

\noindent Steven V Sam\\
Department of Mathematics\\
Massachusetts Institute of Technology\\
Cambridge, MA 02139\\
ssam@mit.edu\\
\url{http://www.mit.edu/~ssam}

\end{document}